\documentclass{amsart}
\usepackage{amsmath,mathrsfs}
\usepackage[all]{xy}

\numberwithin{equation}{section}

\newtheorem{thm}{Theorem}[section]
\newtheorem*{thm1}{Theorem \ref{dam1}}
\newtheorem*{thm2}{Theorem \ref{dam2}}
\newtheorem{claim}[thm]{Claim}
\newtheorem{cor}[thm]{Corollary}
\newtheorem{lem}[thm]{Lemma}
\newtheorem{prop}[thm]{Proposition}
\theoremstyle{definition}
\newtheorem{defn}[thm]{Definition}
\newtheorem{qstn}[thm]{Question}
\theoremstyle{remark}

\newcommand{\N}{\mathbb{N}}
\newcommand{\Q}{\mathbb{Q}}

\newcommand{\M}{\mathscr{M}}

\newcommand{\res}{\upharpoonright}

\newcommand{\ngeq}{\not \geq}

\newcommand{\emp}{\emptyset}

\newcommand{\0}{{\bf 0}}

\newcommand{\RCA}{\mathsf{RCA}_0}
\newcommand{\WKL}{\mathsf{WKL}_0}
\newcommand{\ACA}{\mathsf{ACA}_0}

\newcommand{\RT}{\mathsf{RT}}
\newcommand{\SRT}{\mathsf{SRT}}
\newcommand{\SRAM}{\mathsf{SRAM}}
\newcommand{\ASRT}{\mathsf{ASRT}}
\newcommand{\ASRAM}{\mathsf{ASRAM}}

\newcommand{\DNR}{\mathsf{DNR}}
\newcommand{\COH}{\mathsf{COH}}

\newsymbol\upharpoonright 1316
\newsymbol\varnothing 203F
\newsymbol\nvdash 2330 
\newsymbol\nVdash 2331
\newsymbol\Vdash 130D
\newsymbol\nprec 2306
\newsymbol\npreceq 230E
\newsymbol\nleftrightarrow 233D
\newsymbol\nsubseteq 232A
\newsymbol\nsupseteq 232B

\begin{document}

\title{Stable Ramsey's theorem and measure}
\author{Damir D. Dzhafarov}
\address{University of Chicago}
\email{damir@math.uchicago.edu}
\thanks{{\em $2010$ Mathematics Subject Classification.} Primary 03D80, 05D10, 03D32, 03F35. }
\thanks{The author is grateful to his thesis advisers, Robert Soare, Denis Hirschfeldt, and Antonio Montalb\'{a}n, for valuable insights and helpful comments during the preparation of this article.  He also thanks the anonymous referee for a careful reading of the article, and for numerous comments that helped to improve it.}

\begin{abstract}
The stable Ramsey's theorem for pairs has been the subject of numerous investigations in mathematical logic.
%It asserts that every stable coloring of pairs of integers, i.e. one in which the color of $\{x,y\}$ depends only on $x$ for large enough $y$, has an infinite homogeneous set.
We introduce a weaker form of it by restricting from the class of all stable colorings to subclasses of it that are non-null in a certain effective measure-theoretic sense.  We show that the sets that can compute infinite homogeneous sets for non-null many computable stable colorings and the sets that can compute infinite homogeneous sets for all computable stable colorings agree below $\emp'$ but not in general.  We also answer the analogs of two well known questions about the stable Ramsey's theorem by showing that our weaker principle does not imply $\COH$ or $\WKL$ in the context of reverse mathematics.
\end{abstract}

\maketitle

\section{Introduction}\label{sec_intro}
The logical content of Ramsey's theorem has been studied extensively from the point of view of computability theory, beginning with the work of Jockusch \cite{J}. Previous investigations, a partial survey of which can be found in \cite{CJS}, pp. 5--8, have been primarily concerned with identifying which complexity classes do or do not contain homogeneous sets for all computable colorings, thereby gauging the general difficulty of finding solutions to instances of Ramsey's theorem.

In this article, we concentrate on the {\em stable} form of Ramsey's theorem, which has played an important role in the study of Ramsey's theorem proper.  We restrict our analysis from the class of all stable colorings to ``large'' or non-null subclasses of it, using a notion of nullity for $\Delta^0_2$ sets (see Section \ref{sec_meas}).  A previous result in this direction was obtained by Hirschfeldt and Terwijn \cite[Theorem 3.1]{HT} and appears as Theorem \ref{thm:ht} below.  The focus here is on classifying properties of homogeneous sets of stable colorings not, as above, into those that are and are not universal, but into those that are and are not typical.

We begin by reviewing some of the terminology specific to the study of Ramsey's theorem.  We refer the reader to Soare \cite{So} for general background material on computability theory.

\begin{defn}\label{defn_basic_bkg}
Let $X$ be an infinite subset of $\omega$ and fix $n,k \in \omega$.
\begin{enumerate}
\item $[X]^n$ denotes the set of all subsets of $X$ of cardinality $n$.
\item A {\em $k$-coloring of $[X]^n$} is a map $f : [X]^n \to k$, where $k$ is identified with the set of its predecessors, $\{0,\ldots,k-1\}$.
\item A set $H \subseteq X$ is {\em homogeneous} for $f$ provided $f \res [H]^n$ is constant.
\item If $X = \omega$ and $n = k = 2$, we call $f$ simply a {\em coloring of pairs}, and if in addition $\lim_s f(x,s)$ exists for all $x$ we call $f$ a {\em stable coloring}.
\end{enumerate}
\end{defn}

Ramsey's theorem for pairs, denoted $\RT^2_2$, asserts that every coloring of pairs has an infinite homogeneous set, while the stable Ramsey's theorem, denoted $\SRT^2_2$, makes this assertion only for stable colorings.  Restricting to computable colorings allows for  the study of the effective content of homogeneous sets.  For stable colorings, this reduces via the limit lemma to the study of infinite subsets and cosubsets (i.e., subsets of complements) of $\Delta^0_2$ sets (for details, see \cite{CJS}, Lemma 3.5). In particular, every computable stable coloring has an infinite homogeneous set of degree at most $\0'$, a fact not true of computable colorings in general (\cite{J}, Corollary 3.2).

A natural question then is whether this upper bound can be improved somehow.  With respect to the low$_n$ hierarchy, the following well-known results give a sharp separation.

%Specifically, each $\Delta^0_2$ set $A$ gives rise to a computable stable coloring of which every infinite subset and cosubset of $A$ computes an infinite homogeneous set, and conversely, each such coloring $f$ gives rise to a $\Delta^0_2$ set of which every infinite homogeneous set of $f$ is an infinite subset or cosubset (for details, see \cite[Lemma 3.5]{CJS}.)  

\begin{thm}[Cholak, Jockusch, and Slaman \cite{CJS}, Theorem 3.1]\label{thm_cjs}
Every computable coloring of pairs (not necessarily stable) has a low$_2$ infinite homogeneous set.
\end{thm}

%\noindent However, we cannot improve this from low$_2$ to low$_1$, even for stable coloring.s

\begin{thm}[Downey, Hirschfeldt, Lempp, and Solomon \cite{DHLS}]\label{thm:no_low}
There exists a computable stable coloring with no low infinite homogeneous set.
\end{thm}

\noindent The next result gives instead an improvement over the original bound with respect to the arithmetical hierarchy.

\begin{thm}[Hirschfeldt, Jockusch, Kjos-Hanssen, Lempp, and Slaman \cite{HJKLS}, Corollary 4.6]\label{thm_huge}
Every computable stable coloring has an infinite homogeneous set of degree strictly below $\0'$.
\end{thm}

The above mentioned result of Hirschfeldt and Terwijn from \cite{HT} is a measure-theoretic analysis of Theorem \ref{thm:no_low} and shows that this theorem is atypical in that the collection of computable stable colorings that actually do have a low infinite homogeneous set is not null in the sense of $\Delta^0_2$ nullity.

%With regard to Theorems \ref{thm_cjs} and \ref{thm_huge}, Mileti \cite[Sections 5.3 and 5.4]{M} asked whether they admit sufficiently uniform proofs that there could be a single low$_2$ degree, respectively a single degree strictly below $\0'$, bounding the degree of a homogeneous set for every stable computable coloring.

In this article, we similarly analyze Theorems \ref{thm_cjs} and \ref{thm_huge}.  As both theorems are positive, we turn our attention to uniformity.  Mileti \cite[Theorem 5.3.7 and Corollary 5.4.6]{M} showed that neither of these theorems admits a uniform proof.  In Section \ref{sec_sram}, we extend one of his results by showing the following:

\begin{thm}\label{dam1}
For each $\bf d < 0'$, the class of computable stable colorings having an infinite homogeneous set of degree at most $\bf d$ is $\Delta^0_2$ null.
\end{thm}

%We discuss Mileti's negative answers to these questions in Secton \ref{sec_sram} below, and extend the latter one to show that there is not even a single degree strictly below $\0'$ that works for any ``large'' collection of computable stable colorings.

\noindent In Section 4, we prove the following theorem showing that uniformity results can differ between the class of all computable stable colorings and more general subclasses of it that are not $\Delta^0_2$ null.  The $\Delta^0_3$ bound also gives a partial result in the direction of showing that $<\0'$ in the preceding theorem cannot be replaced by low$_2$.

\begin{thm}\label{dam2}
There is a degree $\bf d \leq 0''$ such that the class of computable stable colorings having an infinite homogeneous set of degree at most $\bf d$ is not $\Delta^0_2$ null but is not equal to the class of all such colorings.
\end{thm}

\noindent In Section \ref{sec_rm}, we introduce several combinatorial principles related to $\SRT^2_2$ from a measure-theoretic viewpoint, and study these in the context of reverse mathematics.  In particular, we introduce the principle $\ASRT^2_2$ which asserts that ``non-negligibly many'', rather than all, computable stable colorings admit a homogeneous set, and show that it lies strictly in between $\SRT^2_2$ and the axiom $\DNR$, and that it does not imply $\WKL$.  For background on reverse mathematics, see Simpson \cite{Si}.

\section{$\Delta^0_2$ measure}\label{sec_meas}
Martin-L\"{o}f introduced the definition of 1-randomness as a constructive notion of nullity.  A stricter approach is that of Schnorr \cite{Sc}, which we now briefly recall.

\begin{defn}
A {\em martingale} is a function $M : 2^{<\omega} \to \mathbb{R}^{\geq0}$ that satisfies, for every $\sigma \in 2^{<\omega}$, the averaging condition
\begin{equation}\label{cond:avg}
2M(\sigma) = M(\sigma0) + M(\sigma1).
\end{equation}
We say that $M$ {\em succeeds} on a set $A$ if $\limsup_{n \to \infty} M(A \res n) = \infty$, and we let the {\em success set} of $M$, $S[M]$, be the class of all sets on which $M$ succeeds.\end{defn}

%\noindent For our purposes, it will suffice to work with {\em rational}-valued martingales (see \cite{Tt}, Proposition 1.5.5), and unless otherwise noted we shall henceforth assume.
%\noindent A {\em computation} of a martingale $M$ is a function $\widetilde{M} : 2^{<\omega} \times \omega \to \Q^{\geq0}$ satisfying $|\widetilde{M}(\sigma,s) - M(\sigma)| \leq 2^{-s}$ for all $\sigma$ and $s$, and $M$ is called {\em computable} if it has a computable computation.  It can be shown (see \cite{Tt}, Proposition 1.5.5) that for every such $M$ there is a rational-valued martingale $M_0$ which is actually computable in the ordinary sense, and such that $S[M] \subseteq S[M_0]$.
\noindent Unless otherwise noted, we shall assume that all our martingales are rational-valued, so that it makes sense to speak of martingales being computable.  A class $\mathscr{C} \subseteq 2^{\omega}$ is said  to be {\em computably null} if there is a computable martingale $M$ which succeeds on each $A \in \mathscr{C}$, and {\em Schnorr null} if in fact there is a computable nondecreasing unbounded function $h$ with $\limsup_{n \to \infty} \frac{M(A \res n)}{h(n)} = \infty$ for every such $A$ (i.e., the martingale succeeds sufficiently fast).
%For details, see the discussion on p. 132 of  \cite{HT}.
The motivation here comes from the following classical result of Ville.  The interested reader may wish to consult \cite{Tt}, Section 1.5, for a thorough treatment of effective measure, and \cite{DH} for background on algorithmic complexity.

\begin{thm}[Ville's theorem]
A class $\mathscr{C} \subseteq 2^{\omega}$ has Lebesgue measure $0$ if and only if there is martingale $M$ such that $\mathscr{C} \subseteq S[M]$.
\end{thm}

By relativizing computable nullity to $\emp'$, we thus obtain a notion of nullity for the class of $\Delta^0_2$ sets.

\begin{defn}\label{defn_meas0}
A class $\mathscr{C} \subseteq 2^{\omega}$ is {\em $\Delta^0_2$ null}  (or has {\em $\Delta^0_2$ measure $0$}) if there exists a $\Delta^0_2$ martingale $M$ such that $\mathscr{C} \subseteq S[M]$.
\end{defn}

\noindent The study of this notion of nullity has been conducted principally by Terwijn \cite{Tt, Tq} and by Terwijn and Hirschfeldt \cite{HT}, though in more general contexts it goes back to Schnorr (see \cite{Sc}, p. 55).  It is a reasonable notion of nullity in that many of the basic properties one would expect to hold, do.

\begin{prop}[Lutz, see \cite{Tt}, Section 1.5]\label{thm:lutzbig}
\
\begin{enumerate}
\item The class of all $\Delta^0_2$ sets is not $\Delta^0_2$ null.
\item For every $\Delta^0_2$ set $A$, $\{A\}$ is $\Delta^0_2$ null.
\item If $\mathscr{C}_0,\mathscr{C}_1,\ldots$ is a sequence of subsets of $2^\omega$ and $M_0,M_1,\ldots$ a uniformly $\Delta^0_2$ sequence of martingales such that $\mathscr{C}_e \subseteq S[M_e]$ for every $e \in \omega$, then $\bigcup_{e \in \omega}\mathscr{C}_e$ is $\Delta^0_2$ null.
\end{enumerate}
\end{prop}

\noindent Additionally, Lutz and Terwijn (see \cite{Tt}, Theorem 6.2.1) have shown that for every $\Delta^0_2$ set $A >_T \emp$, the upper cone $\{B : B \geq_T A\}$ is $\Delta^0_2$ null, thereby effectivizing the corresponding classical result of Sacks for Lebesgue measure.  

In view of the remarks following Definition \ref{defn_basic_bkg}, we can use $\Delta^0_2$ nullity as a reasonable notion of ``smallness'' for computable stable colorings.  It is easy to show that the class of $\Delta^0_2$ sets having an infinite computable subset or cosubset is $\Delta^0_2$ null, meaning that ``most'' stable colorings do not have a computable infinite homogeneous set (it is equally easy to extend this from computable to c.e.\ or even co-c.e.). The following result is an instance where the measure-theoretic approach differs from the classical computability-theoretic one.

\begin{thm}[Hirschfeldt and Terwijn \cite{HT}, Theorem 3.1]\label{thm:ht}The class of low sets is not $\Delta^0_2$ null.
\end{thm}

\noindent In fact, the proof of the above theorem gives the stronger result that the class of $\Delta^0_2$ sets not having an infinite low subset or cosubset is $\Delta^0_2$ null.  It follows that ``most'' computable stable colorings do not satisfy Theorem \ref{thm:no_low}.

We will need a more uniform version of the above theorem, which we present in the form of the proposition below, in our proof of Theorem \ref{dam2} in Section \ref{sec_asram}.  It will rely on the following three facts.  The first is the existence of a universal oracle c.e.\ martingale, i.e., of a real-valued martingale $U$ such that for all sets $X$,  $\{x \in \Q: x <U^X(\sigma)\}$ is $X$-c.e.\ uniformly in $\sigma$, and $S[U^X] = \{B \in 2^{\omega}: B \text{ not $X$-random}\}$ (see, e.g., \cite{DH}, Corollary 5.3.5).  By the proof of Proposition 1.5.5 in \cite{Tt}, we can fix a $u \in \omega$ so that for all $X$, $\Phi^{X'}_u$ is a rational-valued martingale with $S[\Phi^{X'}_u] \supseteq S[U^X]$.  The second, which we will use repeatedly in the sequel, is van Lambalgen's theorem (see \cite{DH}, Theorem 5.9.1), which states that a set is 1-random if and only if its odd and even halves are relatively 1-random.
%$A = B \oplus C$ is 1-random if and only if $B$ and $C$ are relatively $1$-random.
And the third fact, due to Nies and Stephan (unpublished, see \cite{DHMN}, Theorem 3.4), is the following theorem.  Recall that if $\{C_s\}_{s \in \omega}$ is a computable approximation of a $\Delta^0_2$ set, its \emph{modulus of convergence} of is the function $m(x) = (\mu s)(\forall t \geq s)[C_s(x) = C_t(x)]$.  We write $\varphi^X_e$ for the use of a computation $\Phi^X_e$.

\begin{thm}[Nies and Stephan]\label{nies_and_stephan}
Let $C$ and $B$ be sets such that $C$ is $\Delta^0_2$ and $B$-random (i.e., 1-random relative to $B$).  If $m$ is the modulus of convergence of a computable approximation of $C$, then $\varphi^B_x(x) \leq m(x)$ for all large enough $x$ such that $\Phi^{B}_x(x) \downarrow$.  In particular, since $m \leq_T \emp'$, $B$ is GL$_1$ (i.e., $B' \leq_T B \oplus \emp'$).
\end{thm}

%see also \cite{DHMN}, Theorem 3.4).
%that if $B$ and $C$ are sets such that $C$ is $\Delta^0_2$ and $B$-random (i.e., 1-random relative to $B$), then $B$ is GL$_1$.  (Below, we write $\varphi^X_e$ for the use of a computation $\Phi^X_e$.)

%\begin{thm}[Nies and Stephan]
%If $B$ and $C$ are sets such that $C$ is $\Delta^0_2$ and $B$-random (i.e., 1-random relative to $B$), then $B$ is GL$_1$, i.e., $B' \leq_T \emp' \oplus B$.
%\end{thm}
%\begin{proof}
%Fix any computable approximation $\widetilde{C}$ of $C$, and let $m$ be its modulus of convergence, i.e.,
%$$
%m(x) = (\mu s)(\forall t \geq s)[\widetilde{C}(x,s) = \widetilde{C}(x,t)]
%$$
%for all $x$.  For each $e$, let $\widetilde{S}_e$ be the set defined by waiting for the least $s$ such that $\Phi^B_{e}(e)$ converges in $s$ steps, letting $\sigma_0$ be the string $\widetilde{C}(0,s) \cdots \widetilde{C}(e,s)$, and enumerating every $\sigma \succeq \sigma_0$ into $\widetilde{S}_e$.  Clearly, $\widetilde{S}_e$ is $B$-c.e.\ uniformly in $e$ and $\mu(\widetilde{S}_e) \leq$, so if we set $S_e = \bigcup_{i \leq e} \widetilde{S}_i$ then $S_0,S_1,\ldots$ is a Martin-L\"{o}f test relative to $B$.
%\end{proof}

Recall that a {\em $\Delta^0_2$ index} for a $\Delta^0_2$ set $A$ (or, more generally, for a partial $\emp'$-computable function $f$) is an $i \in \omega$ such that $A =  \Phi_i^{\emp'}$ ($f =  \Phi_i^{\emp'}$).  A {\em lowness index} for a low set $L$ is a $\Delta^0_2$ index for $L'$.  We draw attention to our use of $\Phi^X_{e,s}(x)$ to indicate a computation with oracle $X$ run for $s$ steps on input $x$, versus our use of $\Phi^X_{e}(x)[s]$ to indicatee the computation  $\Phi^{X_s}_{e,s}(x)$ under the assumption of a fixed computable approximation (or enumeration) $\{X_s\}_{s \in \omega}$ of $X$.  In particular, determining whether $\Phi^X_{e,s}(x)$ converges is $X$-computable, while for $\Phi^{X_s}_{e,s}(x)$ it is computable.  We fix a computable enumeration $\{\emp'_s\}_{s \in \omega}$ of $\emp'$.

\begin{prop}\label{lem:uniform}
There exists a $\emp''$-computable function $f$ such that for every $e, i \in \omega$, if $\Phi^{\emp'}_e$ is total and a martingale, and if $i$ is a lowness index for some set $L$, then there is a set $B \notin S[\Phi^{\emp'}_e]$ such that $f(e,i)$ is a lowness index for $L \oplus B$.\end{prop}

\begin{proof}
Fix $e,i \in \omega$ and let $u \in \omega$ be as described above.  We define a  partial $\emp'$-computable function $M : 2^{<\omega} \to \Q^{\geq0}$.  Given $\sigma \in 2^{<\omega}$, let $\widetilde \sigma$ be either $\lambda$ if $\sigma = \lambda$, or $\sigma(0)\sigma(2)\cdots\sigma(2m)$ if $\sigma$ has length $2m+1$ or $2m+2$ for some $m \geq 0$.  
%Let $M(\lambda) = 1$, and suppose $\sigma \in 2^{<\omega}$ has length $2m+1$ or $2m+2$ for some $m \geq 0$.
If there exist $q, r \in \Q^{\geq 0}$ and $\tau \in 2^{<\omega}$ such that
\begin{enumerate}
\item $\Phi^{\emp'}_{e}(\widetilde{\sigma})\downarrow = q$,
\item $\Phi^{\emp'}_{i} (x) \downarrow = \tau(x)$ for all $x < |\tau|$ and $\Phi^\tau_u(\sigma) \downarrow = r,$
\end{enumerate}
then let $M(\sigma) = \frac{1}{2}(q+r)$, and otherwise let $M(\sigma)$ be undefined.  It is not difficult to see that $M$ satisfies the averaging condition (\ref{cond:avg}) where defined.

We next define $\{0,1\}$-valued partial $\emp'$-computable functions $A$, $B$, and $C$ as follows.  Given $x$, let
$$A(x) =
\left \{
\begin{array}{ll}
0 & \text{if } M((A \res x)\,0) \downarrow \; \leq M(A \res x) \downarrow\\
1 & \text{if } M((A \res x)\,0) \downarrow \; > M(A \res x) \downarrow\\
\uparrow & \text{otherwise}
\end{array}
\right . .
$$
Then let $B(x) = A(2x)$ and $C(x) = A(2x+1)$ for all $x$, and let $c$ be a $\Delta^0_2$ index for $C$.  Finally, define also $m_C(x) = (\mu s)(\forall t \geq s)[\Phi^{\emp'}_c(x)[t] \downarrow = \Phi^{\emp'}_c(x)[s]\downarrow].$

Notice that if $\Phi^{\emp'}_e$ is a total martingale and $\Phi^{\emp'}_i$ is (the characteristic function of) the jump of some set $L$, then $M$ is a $\Delta^0_2$ martingale whose success set includes that of $\Phi^{L'}_u$, and $A$ is a $\Delta^0_2$ set on which $M$ does not succeed.  We then also have that $A = B \oplus C$, and it is readily seen from the definition of $M$ that $B \notin S[\Phi^{\emp'}_e]$. Now because $A \notin S[M]$, $A$ must be $L$-random, and so by van Lambalgen's theorem relative to $L$, $C$ must be $L\oplus B$-random.  Moreover, $m_C$ is in this case the modulus of convergence for the computable approximation $\{C_s\}_{s \in \omega}$ of $C$ defined by $C_s(x) = i$ if $\Phi^{\emp'}_{c}(x)[s] \downarrow = i$ and $C_s(x) = 0$ otherwise.  Hence, by Theorem \ref{nies_and_stephan} (with $L \oplus B$ in place of $B$), there must be an $n$ so that for all $x \geq n$, whenever $\varphi^{L \oplus B}_x(x)$ is defined it is bounded by $m_C(x)$.

Now to define $f(e,i)$, choose $j \in \omega$ so that $\Phi^{X'}_j = X$ for all sets $X$, and let  $h$ be a computable function so that for all $x \in \omega$, $x \in X$ if and only if $h(x) \in X'$.  Using a $\emp''$ oracle, we search for the first of the following to occur:
\begin{enumerate}
\item $\Phi^{\emp'}_e$ is undefined or does not satisfy the averaging condition (\ref{cond:avg}) on some string,

\item $\Phi^{\emp'}_i$ is undefined on some number,

\item there exist a $\sigma \in 2^{<\omega}$ and an $x < |\sigma|$ such that $\Phi_i^{\emp'}(h(y)) \downarrow = \sigma(y)$ for all $y < |\sigma|$,
and either $\Phi_x^{\sigma}(x) \downarrow$ and $\Phi^{\emp'}_i(x) \downarrow = 0$, or else $\Phi_x^{\tau}(x) \uparrow$ for all $\tau \supseteq \sigma$ and $\Phi^{\emp'}_i(x) \downarrow = 1$,

\item there is an $n \in \omega$ so that for all $\sigma, \tau$ of the same length and all $x \geq n$, if
\begin{enumerate}
\item $\Phi_i^{\emp'}(h(y)) \downarrow = \sigma(y)$ for all $y < |\sigma|$,
\item $B(y) \downarrow = \tau(y)$ for all $y < |\tau|$,
\item $\Phi^{ \sigma \oplus \tau}_x(x) \downarrow$ and $m_C(x) \downarrow$,
\end{enumerate}
then $\varphi^{\sigma \oplus \tau}_x(x) \leq m_C(x)$.

\end{enumerate}
This search necessarily terminates, for if (1), (2), and (3) above do not obtain, then we are precisely in the situation of the preceding paragraph, so (4) must obtain as discussed there.  If (1), (2), or (3) occur, let $f(e,i) = 0$.  Otherwise, choose the least $n$ witnessing the occurrence of (4) and let $f(e,i)$ be a $\Delta^0_2$ index, found according to some fixed effective procedure, for the following function.  On input $x$, the function waits for $m_C(x)$ to converge, then chooses the smallest $y \geq n$ such that $\Phi^{X}_x = \Phi^X_y$ for all sets $X$ and searches for the first $\sigma, \tau$ of the same positive length so that (a) and (b) in (4) above hold.  It then outputs $1$ or $0$ depending as $\Phi^{\sigma\oplus\tau}_y(x) \downarrow$ with use bounded by $m_C(x)$ or not.
\end{proof}

%%%%%%%%%

\section{Almost s-Ramsey degrees}\label {sec_sram}

In \cite[Sections 4 and 5]{CJS}, Cholak, Jockusch, and Slaman give two proofs of Theorem \ref{thm_cjs} for the stable case, but neither of them is uniform over the stable colorings (see the discussion at the beginning of Section 12.3 of \cite{CJS}), and similarly in the case of the proof of Theorem \ref{thm_huge}.  To address whether such nonuniformities were essential, Mileti introduced the following class of degrees:

%Uniform arguments generally carry more information and often find application to other questions (see, for example, \cite{DHLS}, bottom of p. 1372), so it is of interest whether the nonuniform steps in these theorems are essential, or whether they can be circumvented by different methods.
%One natural question emerging from Theorems \ref{thm_cjs} and \ref{thm_huge} is whether they admit uniform proofs, 

\begin{defn}[Mileti \cite{M}, Definition 5.1.2]\label{defn_sRamseydeg}
A Turing degree $\bf d$ is {\em s-Ramsey} if every $\Delta^0_2$ set has an infinite subset or cosubset of degree at most $\bf d$.
\end{defn}

\noindent Obviously, an s-Ramsey degree can also be defined as one which bounds the degree of a homogeneous set for every computable stable coloring.  Thus, the following results imply that Theorems \ref{thm_cjs} and \ref{thm_huge} do not have uniform proofs.

\begin{thm}[Mileti \cite{M}, Theorem 5.3.7 and Corollary 5.4.6]\label{thm:joelims}
\
\begin{enumerate}
\item The only $\Delta^0_2$ s-Ramsey degree is $\0'$.
\item There is no low$_2$ s-Ramsey degree.
\end{enumerate}
\end{thm}

With the definition of $\Delta^0_2$ nullity in hand, we can generalize s-Ramsey degrees by passing from the class of all $\Delta^0_2$ sets to subclasses of it which are not $\Delta^0_2$ null.

%...Comments about almost s-Ramsey degrees....uniformity of proofs for "large classes".

\begin{defn}\label{defn_asRamseydeg}
A Turing degree $\bf d$ is {\em almost s-Ramsey} if the collection of $\Delta^0_2$ sets with an infinite subset or cosubset of degree at most $\bf d$ is not \mbox{$\Delta^0_2$ null}.
\end{defn}

\noindent We obtain the same class of degrees in the above definition whether we insist on considering cosubsets or not. For if a martingale $M$ succeeds on the class of all $\Delta^0_2$ sets having an infinite subset of degree at most $\bf d$, then the martingale $M + N$, where $N(\sigma) = M( (1- \sigma(0))(1-\sigma(1))\cdots(1-\sigma(|\sigma| - 1)))$ for all $\sigma$,  succeeds on the class of all $\Delta^0_2$ sets having an infinite such subset or cosubset.  This is in stark contrast to Definition \ref{defn_sRamseydeg} even if we deal only with infinite, coinfinite $\Delta^0_2$ sets, as it is easy to construct such a set so that all of its infinite subsets compute $\emp'$ (in fact, for any infinite set $A$, if $B$ is the set of all prefixes of $A$ under some fixed computable bijection of $2^{<\omega}$ with $\omega$, then each infinite subset of $B$ computes $A$).
%let $A(x) = 1$ if $x = 2 \langle e, s \rangle$ with $s = (\mu t) [\emp' \res e = \emp'_t \res e]$, and let $A(x) = 0$ otherwise).
%\footnote{Let $p_0<p_1<\cdots$ be a listing of the primes, define $A(x) = 0$ if $x$ is not a prime power, and for $s > 0$ let $A(p_e^s)$ be $1$ or $0$ depending as $s$ is or is not the least stage $t$ with $\emp' \res e = \emp'_t \res e$.}

The preceding definition was suggested by D. Hirschfeldt, who asked whether Mileti's results still hold if s-Ramsey degrees are replaced by the weaker almost s-Ramsey degrees, and more generally, whether the two classes of degrees are the same.  Theorem \ref{dam1}, stated in Section \ref{sec_intro} and restated in terms of almost s-Ramsey degrees below, is an affirmative answer with regards to the analog of Theorem \ref{thm:joelims}~(1).  We discuss the other questions, and give a separation of s-Ramsey and almost s-Ramsey degrees, in the next section.

%Our first result about almost s-Ramsey degrees is a point of similarity with s-Ramsey degrees.

\begin{thm1}
The only $\Delta^0_2$ almost s-Ramsey degree is $\0'$.
\end{thm1}

\begin{proof}
Fix a set $D <_T \emptyset'$.  For each $e \in \omega$, we construct uniformly in $\emp'$ a martingale $M_e$ so as to satisfy the requirement
$$
\begin{array}{lll}
R_e : (\exists^\infty x)(\forall y \leq x)[\Phi^D_e (y) \downarrow \in\{0,1\} \wedge \Phi^D_e(x) = 1] \to (\forall A \supseteq \Phi^D_e)[A \in S[M_e] ].
\end{array}
$$
By Theorem \ref{thm:lutzbig}~(3)---letting $\mathscr{C}_e$ there be $\{A : A \supseteq \Phi^D_e\}$ if $\Phi_e^D$ is a characteristic function and $\emp$ otherwise---this will ensure that the collection of sets containing an infinite subset computable in $D$ is $\Delta^0_2$ null, and hence, by the remarks following Definition \ref{defn_asRamseydeg}, that $\deg(D)$ is not almost s-Ramsey.

Fix a total increasing function $f \leq_T \emp'$ not dominated by any function of degree strictly below $\0'$.  We define $M_e$ by stages, at stage $s$ defining $M_e$ on all strings of length $t$ for a specific $t \geq s$.

\medskip
\noindent \mbox{{\em Stage $s=0$.}}  Let $M_e(\lambda) = 1$.

\medskip
\noindent \mbox{{\em Stage $s+1$.}}  Assume $M_e$ has been defined on all strings of length $t$ for some $t \geq s$.  Search $\emp'$-computably for a string $\tau \subseteq D$ and a number $x \geq t$ such that $|\tau|, x \leq f(t)$ and \mbox{$\Phi^{\tau}_{e,|\tau|}(x) \downarrow = 1$}.  If the search succeeds, choose the least $x$ for which it does so.  Then for each $\sigma \in 2^{<\omega}$ of length $t$, and for all $\tau \supset \sigma$ with $|\tau| \leq x + 1$, define
$$
M_e(\tau) = \left\{
\begin{array}{ll}
M_e(\sigma) & \text{if } |\tau| \leq x\\
2M_e(\sigma) & \text{if } |\tau| = x+1 \wedge \tau(x) = 1\\
0 & \text{if } |\tau| = x+1 \wedge \tau(x) = 0
\end{array}
\right. .
$$
Otherwise, set $M_e(\sigma0) = M_e(\sigma1) = M_e(\sigma)$ for all $\sigma$ of length $t$.

\medskip
It is clear that the construction succeeds in defining $M_e$ on all of $2^{<\omega}$.
To verify that $R_e$ is met, suppose that $\Phi^D_e$ is the characteristic function of an infinite set.  Then the function
$$g(y) = (\mu s)(\exists x \geq y)(\forall z < x)[\Phi^D_{e,s} (x) \downarrow = 1 \wedge (y \leq z \to \Phi^D_{e,s}(z) \downarrow = 0)]$$
is total and computable in $D$, so by choice of $f$ there must exist infinitely many $y$ such that $g(y) \leq f(y)$.  Fix $A \supseteq \Phi^D_e$ and suppose that at the end of some stage $s'$ of the construction, $M_e(A \res t)$ for some $t \geq 0$ is defined and positive, while \mbox{$M_e(A \res t+1)$} is not yet defined.  Choose the least $y \geq t$ such that $g(y) \leq f(y)$.  If $f$ is replaced by $g$ in the search performed at each stage of the construction, then the search always succeeds, so it must necessarily succeed at some stage $s > s'$.
%with $s' < s \leq s' + (y+1 - t)$.
Fix the least such $s$.  Then by construction, at every stage between $s'$ and $s$, $M_e$ gets defined only on the successors of the longest strings it was defined on at the previous stage, and it is given the same value on these successors.  In particular, at the beginning of stage $s$, we have that $M_e$ is defined on $A \res t + (s - s') - 1$ at the start of stage $s$, and $M_e(A \res t + (s - s') - 1) = M_e(A \res t)$.  By choice of $s$, there exists a string $\tau \subseteq D$ and a number $x \geq t + (s-s') -1$ such that $|\tau|, x \leq f(t)$ and \mbox{$\Phi^{\tau}_{e,|\tau|}(x) \downarrow = 1$}.  Then at stage $s$, $M_e$ gets defined on $A \res x+1$ with $M_e(A \res y) = M_e(A \res t)$ for all $y \leq x$ and, since $A(x) = \Phi^D_e(x) = 1$, $M_e(A \res x+1) = 2M_e(A \res t)$.  Since $x + 1 > t$, it follows that $\limsup_n M_e (A \res n) = \infty$.
\end{proof}

%%%% Spectra

We illustrate an application of the preceding theorem by briefly looking at the Muchnik degrees of classes of infinite subsets and cosubsets of $\Delta^0_2$ sets. Recall that if $\mathscr A$ and $\mathscr B$ are classes of sets, we say $\mathscr{A}$ is Muchnik (or weakly) reducible to $\mathscr{B}$, and write $\mathscr{ A} \leq_w \mathscr{B}$, if every element of $\mathscr B$ computes an element of $\mathscr A$; if also $\mathscr{ B} \leq_w \mathscr{A}$, we write $\mathscr{ A} \equiv_w \mathscr{B}$.  We refer the reader to Binns and Simpson \cite{BS}, Section 1, for additional background.

\begin{defn}
Given a $\Delta^0_2$ set $A$, let $H(A)$ be the collection of all infinite subsets or cosubsets of $A$, and for a class $\mathscr{C}$ of $\Delta^0_2$ sets let $H(\mathscr{C})$ denote the structure $\{H(A) : A \in \mathscr{C}\}$ under $\leq_w$.  Given a computable stable coloring $f$, let $H(f)$ be the collection of all infinite homogeneous sets of $f$.
\end{defn}

Clearly, for each $\Delta^0_2$ set $A$ there is a computable stable $f$ with $H(A) \equiv_w H(f)$, and conversely.  Thus, we may use the two notions interchangeably here.

\begin{prop}\label{prop_joins}
$H(\Delta^0_2)$ is a lower semilattice.
\end{prop}

\begin{proof}
Given two stable colorings, $f_0$ and $f_1$, we define a third, $f$, such that $H(f) \equiv_w H(f_0) \cup H(f_1)$.  For $x,y \in \omega$, let $f(2x,y)$ equal $f_0(x,z)$ for the least $z$ such that $2z \geq y$, and let $f(2x+1,y)$ equal $f_1(x,z)$ for the least $z$ such that $2z+1 \geq y$. 
%$$
%f(x,y) =
%\left\{
%\begin{array}{ll}
 %f_1(\frac{x}{2},\frac{z}{2}) & \text{if $x$ is even and $z$ is the least even number $\geq y$}   \\
 %\noalign{\smallskip}
 %f_2(\frac{x-1}{2},\frac{z-1}{2}) &  \text{if $x$ is odd and $z$ is the least odd number $\geq y$}
%\end{array}
%\right..
%$$
It is easy to see that $f$ is stable.

If $H$ is an infinite homogeneous set for $f_0$, respectively for $f_1$, then the set $\{ 2x : x \in H \}$, respectively $\{ 2x+1: x \in H\}$, is homogeneous for $f$, implying that $H(f) \leq_w H(f_0) \cup H(f_1)$.  Conversely, let $H$ be any infinite homogeneous set for $f$ and let $H_0 = \{x : 2x \in H \}$ and $H_1 =  \{x : 2x+1 \in H \}$.  One of $H_0$ and $H_1$, say $H_i$, must be infinite, and this set is clearly computable in $H$ and homogeneous for $f_i$, implying that $H(f_0) \cup H(f_1) \leq_w H(f)$.
\end{proof}

Notice that if there were a largest element in $H(\Delta^0_2)$, it would have an infinite homogeneous set $H <_T \emp'$ by Theorem \ref{thm_huge}.  Then $\deg(H)$ would be an s-Ramsey degree $< {\bf 0'}$, contrary to part (1) of Theorem \ref{thm:joelims}.  This yields the following:

\begin{cor}[Mileti \cite{M}, Corollary 5.4.8]
There is no largest element in $H(\Delta^0_2)$.
\end{cor}

\noindent Using Theorem \ref{dam1}, we can now extend this result as follows.

\begin{cor}
If $\mathscr{C}$ is a class of $\Delta^0_2$ sets that is not $\Delta^0_2$ null, then there is no largest element in $H(\mathscr{C})$.
\end{cor}

For general interest, we remark that the algebraic properties of the structure $H(\Delta^0_2)$ have not previously been studied.  It can be shown, though we do not elaborate on it here, that there are no maximal elements in it, and that for every finite collection of elements in it there is an element incomparable with each of them (proofs will appear in \cite{Dt}).  Beyond this, little is known; in particular, we do not know the answer to the following question:

\begin{qstn}
Is $H(\Delta^0_2)$ elementarily equivalent to $H(\mathscr{C})$ for every class $\mathscr{C}$ of $\Delta^0_2$ sets that is not $\Delta^0_2$ null?
\end{qstn}

%%%%%%%%%%%%%%%%%

\section{An almost s-Ramsey degree that is not s-Ramsey}\label{sec_asram}

In this section, we give a proof of Theorem \ref{dam2}, restated equivalently below, thereby showing that the s-Ramsey degrees are a proper subclass of the almost s-Ramsey degrees.  We do not know whether the analog of Theorem \ref{thm:joelims}~(2) holds for almost s-Ramsey degrees, but as every low$_2$ degree is $\Delta^0_3$, our result is a partial step towards a negative answer.

\begin{thm2}
There is a $\Delta^0_3$ almost s-Ramsey degree that is not s-Ramsey.
\end{thm2}

\begin{proof}
Fix a $\Delta^0_2$ set $A$ with no low infinite subset or cosubset.  Computably in $\emp''$, we construct a set $D$ and infinite low sets $L_0,L_1,\ldots$ that satisfy, for every $e \in \omega$ and $i < 2$, the requirements
$$
\begin{array}{lll}
R_e & : & L_e \times \{e\} =^* D^{[e]} \wedge (\Phi^{\emp'}_e \text{ is a total martingale} \to  L_e \notin S[\Phi^{\emp'}_e]),\\
\noalign{\medskip}
S_{e,i} & : &  \Phi^D_e \text{ is total, $\{0,1\}$-valued and infinite} \to (\exists x)[\Phi^D_e(x) = 1 \wedge A(x) = i].
\end{array}
$$
The first set of requirements ensures that $\{L_e : e \in \omega\}$ is not $\Delta^0_2$ null and that $L_e \leq_T D$ for all $e$, and the second that no infinite subset or cosubset of $A$ is computable in $D$.  Hence, $\deg(D)$ will be the desired degree.

We let $D = \bigcup_s D_s$, where $D_0,D_1,\ldots$ are constructed in stages as follows.  At stage $s$, we define a finite set $D_s$, a function $f_s$ with domain $\omega$, and for each $e$ a restraint $r_{e,s}$.  We also declare each requirement either {\em online} or {\em offline}.  Let  $h$ be a computable function such that for all sets $X$ and all $x \in \omega$, $x \in X$ if and only if $h(x) \in X'$.

\medskip
\noindent {\em Construction.} 

\medskip
\noindent \mbox{{\em Stage $s = 0$.}}  Set $D_0 = \emp$, and $f_0(e) = r_{e,s} = 0$ for all $e$.  Declare all requirements $R_e$ and $S_{e,i}$ for $e \in \omega$ and $i < 2$ online.

\medskip
\noindent \mbox{{\em Stage $s+1$.}}  Let $D_s$, $f_s$, and $r_{0,s},r_{1,s},\ldots$ be given.  Assume inductively that cofinitely many requirements are still online, and that the value of $f_s$ is $0$ on cofinitely many arguments.

\medskip
\noindent \mbox{{\em Case 1: $s+1 \equiv 0\; \rm{mod}\; 3$ or $s+1 \equiv 1\; \rm{mod}\; 3$.}}   Suppose  $s+1 = 3\langle e, j \rangle + i$, where $e, j \in \omega$ and $i < 2$.  If $S_{e,i}$ is online, ask whether there exists an $x \in \omega$ and a finite set $F$ such that
\begin{enumerate}

\item $D_s \subseteq F \subseteq D_s \cup \{\langle y, e' \rangle  \geq r_{e,s} : e' \leq e \to R_{e'} \text{ online}\}$,

\item $\Phi^F_e(x) \downarrow = 1$ and $A(x) = i$,

\item for $e' \leq e$ with $R_{e'}$ online and all $\langle y, e' \rangle \leq \max F \cup \{z : z \leq  \varphi^F_{e}(x)\}$, $\Phi^{\emp'}_{f_s(e')}(h(2y+1)) \downarrow $, and if \mbox{$\langle y, e' \rangle \in F - D_s$} then $\Phi^{\emp'}_{f_s(e')}(h(2y+1)) = 1$.

\item for $e' \leq e$ with $R_{e'}$ online and all $\langle y,e' \rangle \leq \varphi^F_{e}(x)$, if \mbox{$\langle y, e' \rangle \notin F - D_s$} then $\Phi^{\emp'}_{f_s(e')}(h(2y+1))  = 0$.

\end{enumerate}
If so, we find the first such $F$ and $x$ in some fixed enumeration, set $D_{s+1} = F$, let $r_{e',s+1} = r_{e',s}$ for $e' < e$, and let $r_{e',s+1}$ be the least number greater than \mbox{$\max \{ r_{e'',s}: e \leq e'' \leq e'\}$} and $\varphi^F_{e}(x)$ for $e' \geq e$.  We say that $S_{e,i}$ {\em acts} at stage $s+1$, declare it offline, and declare all $S_{e',i}$ with $e' > e$ currently offline online again.  Otherwise, or if $S_{e,i}$ is already offline, we set $D_{s+1} = D_s$ and $r_{e',s+1} = r_{e',s}$ for all $e'$.  Either way, we let $f_{s+1} = f_s$.  Notice that the question of whether or not $x$ and $F$ in Case 1 exist is $\Sigma^{0,{\emp'}}_1$, and hence can be answered by $\emp''$.

\medskip
\noindent \mbox{{\em Case 2: $s+1 \equiv 2\; \rm{mod}\; 3$.}} We begin by choosing the least $e$ such that $R_e$ is online and $f_s(e') = 0$ for all $e' \geq e$, which must exist by inductive hypothesis.  Set $r_{e',s+1} = r_{e',s}$ for all $e'$.  Fix $e' \in \omega$ and assume we have defined $f_{s+1}$ on all $e'' < e'$.  If $e' > e$ or if $R_{e'}$ is offline, set $f_{s+1}(e') = 0$.  Otherwise, let $i$ be either a fixed lowness index for $\emp$ if there is no $e'' < e'$ such that $R_{e''}$ is online, or else $f_{s+1}(e'')$ for the greatest such $e''$.  Then let $f_{s+1}(e')$ be the result of applying to $e'$ and $i$ the $\emp''$-computable function asserted to exist by Proposition \ref{lem:uniform}.

To define $D_{s+1}$, begin by letting $D_{s+1}^{[e']} = D_s^{[e']}$ for all $e'$ such that at least one of the following holds:
\begin{enumerate}
\item $e' > e$,
\item $R_{e'}$ is offline,
\item $\Phi^{\emp'}_{f_{s+1}(e')}$ is not defined or not $\{0,1\}$-valued on $h(2x+1)$ for some $x \leq s$,
\item\label{item:dead} $\Phi^{\emp'}_{e'}$ is not defined or does not satisfy the averaging condition (\ref{cond:avg}) on some string of length $\leq s$,
\end{enumerate}
For all $e'$ for which $(\ref{item:dead})$ obtains, declare $R_{e'}$ offline, and declare all offline $S_{e'',i}$ requirements for $e'' \geq e'$ online.  For all $e'$ such that none of the above obtain, let $D_{s+1}^{[e']} = D_s^{[e']} \cup \{\langle x, e' \rangle > r_{e',s+1} : x \leq s \wedge \Phi^{\emp'}_{f_{s+1}(e')}(h(2x+1)) \downarrow = 1\}$.

\medskip
In either case above only finitely many requirements are declared offline, and $f_{s+1}$ is defined to be positive on only finitely many elements.  Thus, the induction can continue.

\medskip
\noindent \mbox{{\em End construction.}}

\medskip
The entire construction can be performed using a $\emp''$ oracle, hence $D \leq_T \emp''$.  We now verify that all requirements are satisfied.  To begin, note that each $R$ requirement can only switch from being online to being offline but not back, and each $S_{e,i}$ requirement, once offline, can only become online again because some $R_{e'}$ requirement with $e' \leq e$ has become offline.  In particular, each $S$ requirement acts at most finitely many times. Since for every $e$, $r_{e,s}$ is a nondecreasing function in $s$ that increases only when some $S_{e',i}$ with $e' \leq e$ acts, $\lim_s r_{e,s}$ exists.

\begin{claim}\label{claim:tech}
For every $e \in \omega$, $f(e) = \lim_s f_s(e)$ exists.  Moreover, if $R_e$ is permanently online then $f(e)$ is a lowness index, and if $R_e$ is not permanently online then $f(e)  = 0$ and $D^{[e]}$ is finite.
\end{claim}

\begin{proof}
Fix $e \in \omega$ and assume the claim holds for all $e' < e$.  Fix a stage $s \geq 0$ such that for all $e' \leq e$ and all $i < 2$,
\begin{enumerate}
\item if $e' < e$ then $f(e') \downarrow = f_{t}(e')$ for all $t > s$,
\item if $R_{e'}$ is cofinitely often offline, then it is offline at all stages $t \geq s$,
\item if $S_{e',i}$ is cofinitely often offline, then it is offline at all stages $t \geq s$.
\end{enumerate}
First suppose $R_{e}$ is online at stage $s$, and hence permanently thereafter.  Since $0$ is not a lowness index (we assume an enumeration of oracle machines, such as the standard one based on G\"{o}del numberings, that makes this true), the inductive hypothesis implies that at any stage $t \geq s$ that is congruent to $2$ modulo $3$, the number chosen at the beginning of Case 2 of the construction is at least as big as $e$.  Hence, we see from the construction that the value of $f_t(e)$ at any stage $t \geq s$ depends only on $e$ and, if there is an $R_{e'}$ with $e' < e$ which is online at stage $s$, on $f_t(e') = f(e')$ for the largest such $e'$.  Thus $f_{t}(e)$ has the same value for all $t \geq s$, so $f(e) = f_{s}(e)$.

As $R_e$ is never declared offline, it must be that condition (4) in Case 2 of the construction never occurs, and hence that $\Phi^{\emp'}_e$ is a total martingale.  Let $L$ be either $\emp$ or, if there exists an $e' < e$ with $R_{e'}$ permanently online, $\Phi^{\emp'}_{f(e')}$ for the greatest such $e'$.  Then it follows by construction and by Proposition \ref{lem:uniform} that $f(e)$ is a lowness index for $L \oplus B$, where $B$ is a set not in $S[\Phi^{\emp'}_e]$.  In particular, $f(e)$ is a lowness index, as desired.

Now suppose $R_{e}$ is offline at stage $s$.  Then $f_t(e)$ is defined to be $0$ at all stages $t \geq s$, so $f(e) = 0$.  Now no elements can be put into $D^{[e]}_t$ at any stage $t > s$ under Case 1 of the construction, because by condition (1) in that case this can only be done because of the action of some requirement $S_{e',i}$ with $e' \leq e$, and all such requirements have stopped acting by stage $s$.  Moreover, no elements can be put into $D^{[e]}_t$ under Case 2, because condition (2) in that case allows this only when $R_e$ is still online.  Hence, $D^{[e]}_t = D^{[e]}_{s}$ for all $t \geq s$, and so $D^{[e]}$ is finite.
\end{proof}

\begin{claim}
For every $e \in \omega$, requirement $R_e$ is satisfied via a set $L_e$ such that $\bigoplus_{e' \leq e} L_{e'}$ is low.
\end{claim}

\begin{proof}

First suppose that $\Phi^{\emp'}_e$ is a total martingale.  Then condition (4) in Case 2 of the construction never occurs and $R_e$ is online at all stages.  Let $L$ be as in the proof of the preceding claim, and let $L_e$ be the set $B$ from there, so that $f(e)$ is a lowness index for $L \oplus L_e$ and $L_e \notin S[\Phi^{\emp'}_e]$.

It then remains only to show that $L_e \times \{e\}=^* D^{[e]}$.  Let $s$ be a stage as in the proof of the preceding claim.  Since no $S_{e',i}$ requirement with $e' \leq e$ can act at any stage $t \geq s$, it follows by condition (3) in Case 1 of the construction, as well as the fact that $L_e = \{ x : \Phi^{\emp'}_{f(e)}(h(2x+1))\downarrow = 1 \}$, that any element put into $D^{[e]}_t$ for the sake of an $S$ requirement must belong to $L_e \times \{e\}$.  For the same reason we must have that $r_e = r_{e,t}$ for any stage $t \geq s$, and, as mentioned in the previous claim, the number chosen at the beginning of Case 2 of the construction at any such stage $t$ cannot be smaller than $e$.  Hence, at the end of every stage $t \geq s$ that is congruent to $2$ modulo $3$, all elements $x$ in $L_e \times \{e\}$ with $r_ e < x \leq t$ are put into $D^{[e]}_t$.   It follows that $\{x \in D^{[e]} : x > \max D^{[e]}_{s} \} \subseteq L_e \times \{e\}$ and $\{x \in L_e \times \{e\} : x > r_e \} \subseteq D^{[e]}$, which yields the desired result.

Next suppose that $\Phi^{\emp'}_e$ is not a total martingale.  Then at some stage, condition (4) in Case 2 of the construction occurs and $R_e$ is declared offline.  By the previous claim, $D^{[e]}$ is finite, so if we let $L_e = \emp$ then $L_e$ is low and requirement $R_e$ is met.

 Finally, given $e$ let $e_0 < e_1 < \cdots < e_n $ be a listing of all $e' \leq e$ such that $R_{e'}$ is online at stage $s$.  Then $\bigoplus_{j \leq n} L_{e_j}$ is low, for $f(e_0)$ is a lowness index for $\emp \oplus L_{e_0}$, $f(e_1)$ is a lowness index for $(\emp \oplus L_{e_0}) \oplus L_{e_1}$, and so on.  Hence $\bigoplus_{e' \leq e} L_{e'}$ is low since $L_{e'} = \emp$ for all $e' \neq e_j$ for any $j$, and this completes the proof.
 \end{proof}
 
 \begin{claim}
 For every $e \in \omega$ and $i < 2$, $S_{e,i}$ is satisfied.
 \end{claim}
 
 \begin{proof}
Fix $e$ and $i$ and assume inductively that the claim holds for all $e' < e$.  Fix a stage $s \geq 0$ congruent to $i$ modulo $3$ such that for all $e' \leq e$, $f_s(e') = f(e)$ and $D^{[e']}_s=D^{[e']}$ if $R_{e'}$ is not permanently online, and for all $e' < e$, $r_{e',s} = r_e$ and no $S_{e',i}$ requirement with $e' < e$ acts at or after stage $s$.  Assume further that $\Phi^{D}_e$ is total, $\{0,1\}$-valued, and infinitely often takes the value $1$, as otherwise $S_{e,i}$ is satisfied trivially.  Since $L_{e'} \times \{e'\}=^* D^{[e']}$ for all $e' \leq e$, it follows by the previous claim that $\bigcup_{e' \leq e} D^{[e']}$ is low, and since $D_s$ is finite, also that $\bigcup_{e' \leq e} D^{[e']}\cup D_s$ is low.

Now there must exist an $x \in \omega$ and a finite set $F$ such that $A(x) = i$ and such that the following conditions hold:
\begin{enumerate}

\item $D_s \subseteq F \subseteq D_s \cup \{\langle y, e' \rangle  \geq r_{e,s} : e' \leq e \to R_{e'} \text{ online}\}$,
\item $\Phi^{F}_e(x) \downarrow = 1$,
\item for all $e' \leq e$, $F^{[e']} \subseteq D^{[e']}$,
\item  for all $e' \leq e$,  \mbox{$F^{[e']} \res \varphi^{F}_e(x)) = D^{[e']} \res \varphi^{F}_e(x)$}.
\end{enumerate}
Indeed, from our assumptions about $\Phi^D_e$ it follows that there exist arbitrarily large numbers $x$ and corresponding finite sets $F$ satisfying (1)--(4) above, for example all sufficiently long initial segments of $D$.  And we can clearly find such $x$ and $F$ computably in $\bigcup_{e' \leq e} D^{[e']}\cup D_s$.  Hence, if $A(x)$ were equal to $1-i$ for all such $x$, then depending as $i$ is $0$ or $1$, $\bigcup_{e' \leq e} D^{[e']}\cup D_s$ could compute an infinite subset or infinite cosubset of $A$, contradicting that $A$ has no low infinite subset or cosubset.

By choice of $s$, it is easily seen that for all $e' \leq e$, all elements in $D^{[e']} - D_s$ belong to $L_{e'} \times \{e'\}$.  It follows that the question about an $x \in \omega$ and a finite set $F$ asked at stage $s$ of the construction is precisely the question of whether there exist $x$ and $F$ satisfying the conditions above, and as such must have an affirmative answer.  Hence $S_{e,i}$ acts, meaning that for some such $x$ and $F$, $D_{s+1} = F$ and $r_{e',t}$ is greater than $\varphi^F_e(x)$ for all $t > s$ and all $e' \geq e$.  No requirements can then ever put into $D_t$ any elements below $\varphi^F_e(x)$ at any stage $t > s$, meaning that the $\Phi^F_e(x)$ computation is preserved and so $\Phi^D_e(x) = 1$. Consequently, requirement $S_{e,i}$ is satisfied.
 \end{proof}
 \end{proof}
 
\begin{qstn}\label{q:low2}
Does there exist a low$_2$ almost s-Ramsey degree?
\end{qstn}
 
 \section{Almost stable Ramsey's theorem}\label{sec_rm}
 
The proof-theoretic strength of $\SRT^2_2$, as a principle of second order arithmetic, was first studied by Cholak, Jockusch, and Slaman (\cite{CJS}, Sections 7 and 10).
%We recall the following principles that have been studied in conjunction with $\SRT^2_2$.
One major open problem is whether $\SRT^2_2$ implies $\WKL$ over $\RCA$ (see \cite{CJS}, p. 53), the closest related result being by Hirschfeldt, et al. \cite[Theorem 2.4]{HJKLS} that $\SRT^2_2$ implies the weaker axiom $\DNR$.  (That $\WKL$ does not imply $\SRT^2_2$ is by \cite{CJS}, Theorems 11.1 and 11.4; it can also be seen by Theorem \ref{thm:no_low} and the fact that $\WKL$ has a model consisting entirely of low sets).  Another question is whether $\SRT^2_2$ implies $\COH$, which is equivalent by Theorem 1.3 of \cite{CJS} and the correction given in section A.1 of \cite{M} to the question of whether $\SRT^2_2$ implies $\RT^2_2$.  For completeness, we recall the definitions of $\DNR$ and $\COH$.
\begin{defn}The following definitions are made in $\RCA$.
\begin{enumerate}
\item $\COH$ is the statement that for every sequence $\langle X_i : i \in \N \rangle$ of sets, there is an infinite set $X$ such that for every $i \in \N$, either $X \subseteq^* X_i$ or $X \subseteq^* \overline{X_i}$.
\item $\DNR$ is the statement that for every set $X$ there exists a function $f$ that is DNR$^X$, i.e., such that for all $e \in \N$, $f(e) \neq \Phi^X_e(e)$.
%\item $\mathsf{B}\Gamma$, where $\Gamma$ is a set of formulas in two free number variables, is the collection of all statements
%$$
%\forall n[(\forall x < n)(\exists y)\varphi(x,y) \to (\exists m)(\forall x < n)(\exists y < m)\varphi(x,y)]
%$$
%for $\varphi \in \Gamma$.
%\item $\mathsf{B}\Pi^0_1$ is the collection of all statements of the form
%$$
%\forall n[(\forall x < n)(\exists y)\varphi(x,y) \to (\exists m)(\forall x < n)(\exists y < m)\varphi(x,y)],
%$$
%where $\varphi$ is a $\Pi^0_1$ formula.
\end{enumerate}
\end{defn}
%See, for example, Section 1 of \cite{HJKLS} for the definitions of $\DNR$ and $\COH$.

In this section, we study several principles inspired by our investigations above and related to $\SRT^2_2$ by means of a formal notion of $\Delta^0_2$ nullity.
%To begin, we need to formalize the concept of $\Delta^0_2$ martingales in $\RCA$.

\begin{defn}\label{defn:rmmart}The following definitions are made in $\RCA$.
\begin{enumerate}
\item A {\em martingale approximation} is a function $M : 2^{<\N} \times \N \to \Q^{\geq0}$ such that $\lim_s M(\sigma,s)$ exists for every $\sigma \in 2^{<\N}$ (i.e., $M(\sigma,s) = M(\sigma,t)$ for all sufficiently large $s,t \in \N$), and for all $s \in \N$,
$$
2 M(\sigma,s) = M(\sigma0,s) + M(\sigma1,s).
$$
\item We say $M$ {\em succeeds on} a stable coloring $f : [\N]^2 \to 2$ if
\begin{equation}\label{eqtn:rmmartcond}
(\forall n)(\exists \sigma)(\exists s)(\forall t \geq s)(\forall x < |\sigma|)[\sigma(x) = f(x,t) \wedge M(\sigma,t) = M(\sigma,s) > n].
\end{equation}
\end{enumerate}
\end{defn}

\noindent We can now state an ``almost stable Ramsey's theorem'', along with principles asserting the existence of s-Ramsey and almost s-Ramsey degrees.

\begin{defn}\label{defn:measprinc}The following definitions are made in $\RCA$.
\begin{enumerate}
\item $\ASRT^2_2$ is the statement that for every martingale approximation $M$, there is a stable coloring $f \leq_T M$ on which $M$ does not succeed and which has an infinite homogeneous set.

\item $\SRAM$ is the statement that for every set $X$, there is a set $Y$ as follows: every stable coloring $f\leq_T X$ has an infinite homogeneous set $H \leq_T Y$.

\item $\ASRAM$ is the statement that for every set $X$, there is a set $Y$ as follows: for every martingale approximation $M \leq_T X$ there is a stable coloring $f\leq_T X$ on which $M$ does not succeed and which has an infinite homogeneous set $H \leq_T Y$.
\end{enumerate}
\end{defn}

\noindent Notice that the class of $\Delta^0_2$ sets having an infinite subset or cosubset in a given $\omega$-model of $\ASRT^2_2$ is not $\Delta^0_2$ null.

We begin with the following formalization of Proposition \ref{thm:lutzbig}~(1).  Recall that $\mathsf{B}\Pi^0_1$ is the collection of all statements of the form$$\forall n[(\forall x < n)(\exists y)\varphi(x,y) \to (\exists m)(\forall x < n)(\exists y < m)\varphi(x,y)],$$where $\varphi$ is a $\Pi^0_1$ formula (we do not know if its use below can be avoided).
%For a definition of $\mathsf{B}\Pi^0_1$, and bounding schemes in general, see \cite{CJS}, Definition 6.2(vii).

\begin{lem}[$\RCA + \mathsf{B}\Pi^0_1$]
For every martingale approximation $M$, there is a stable coloring $f \leq_T M$ on which $M$ does not succeed.
\end{lem}

\begin{proof}
Let $M$ be a martingale approximation, say with $\lim_s M(\lambda, s) = 1$.  Then by Definition \ref{defn:rmmart}, if $M(\sigma,s) \leq 1$ for some $s$, either $M(\sigma0,s) \leq 1$ or $M(\sigma1,s) \leq 1$.  Choose $s_0$ so that $M(\lambda, s) \leq 1$ for all $s \geq s_0$.  For every $x$ and $s \geq s_0$, a simple $\Sigma^0_0$ induction then shows that there exists $\sigma \in 2^{<\N}$ of length $x+1$ such that
$$
 (\forall y \leq x+1)[M(\sigma \res y, s) \leq 1] \, \wedge \,(\forall y \leq x)[\sigma(y) = 1 \to M((\sigma \res y) \, 0, s) > 1 ],
$$
and that this string is unique.  Define $f : [\N]^2 \to 2$ by letting $f(x,s)$ for $x < s$ be $0$ or $\sigma(x)$ for the above $\sigma$ depending as $s < s_0$ or $s \geq s_0$.  Clearly, $f $ has a $\Sigma^0_0$ definition with $M$ as parameter, so $f \leq_T M$.  We claim that $f$ is stable and that $M$ does not succeed on it.  Fix $x$ in $\N$ and using $\mathsf{B}\Pi^0_1$ choose an $s\geq s_0$ with $M(\sigma,t) = M(\sigma,s)$ for all $t \geq s$ and $\sigma \in 2^{<\N}$ of length $\leq x + 1$.  Then the $\sigma$ used to define $f(x,s)$ will be same as that used to define $f(x,t)$ for all $t \geq s$.  Hence, $f(x,t) = \sigma(x)$ for all $t \geq s$, and as $M(\sigma, t) \leq 1$ we have the negation of (\ref{eqtn:rmmartcond}) holding with $n = 1$.
\end{proof}

Basic relations of implication and nonimplication between $\SRT^2_2$ and the principles given in Definition \ref{defn:measprinc} are established in the next proposition.

\begin{prop}Over $\RCA$,
\
\begin{enumerate}
\item $\ACA \to \SRAM \to \SRT^2_2 \to \ASRT^2_2$ and $\SRAM \to \ASRAM \to \ASRT^2_2$,
\item $\SRAM$ does not imply $\ACA$, and $\SRT^2_2$ does not imply $\SRAM$.
\end{enumerate}
\end{prop}

\begin{proof}
Clearly, $\SRAM \to \SRT^2_2$ and $\ASRAM \to \ASRT^2_2$.  As for the implications $\SRT^2_2 \to \ASRT^2_2$ and $\SRAM \to \ASRAM$, these follow from the preceding lemma and the fact that $\SRT^2_2$, and hence also $\SRAM$, implies $\mathsf{B}\Pi^0_1$ (\cite{CJS}, comments after Definition 6.4, and Lemma 10.6).
%Thus to prove (1) we only elaborate on the implication $\ACA \to \SRAM$. FINISH
That $\ACA \to \SRAM$ amounts to a formalization of the fact that $\0'$ is an s-Ramsey degree, and is straightforward.

%Fix $X$, and by arithmetical comprehension let
%$$
%\begin{array}{ccl}
%T_i & = & \{ e \in \N : (\forall x)(\forall s)[\Phi^X_e(\langle x, s \rangle) \downarrow \in \{0,1\}]\wedge \, (\forall x)[\lim_s \Phi^X_e(\langle x, s \rangle) \downarrow]\\
%\noalign{\smallskip}
%& & \wedge \, (\forall x)(\exists y > x)[\lim_s \Phi^X_e(\langle y, s \rangle) = i]\}
%\end{array}
%$$
%and $D_i = \{\langle x, e \rangle : e \in T \wedge \lim_s \Phi^X_e(\langle x, s \rangle) = i\}$ for each $i < 2$.  Write $D_{i,e}$ for $\{x : \langle x, e %\rangle \in D_i\}$.
%Now define $g_i : \N^2 \to \N$ as follows.  Let $g_i(x,e) = 0$ for all $x$ and all $e \notin T_i$.  Given $e \in T_i$ and $x \in \N$, there exists a unique $\sigma \in 2^{<\N}$ of length $x+1$ with $\sigma(0) = \min D_{i,e}$ and
%$$
%(\forall y < x)[\sigma(y+1) = (\mu s \in D_{i,e})(\forall z \leq y)(\forall t \geq s)[s > \sigma(y) \wedge f(z,t) = f(z,s) ]].
%$$
%This follows by an easy arithmetical induction on $x$, using $\mathsf{B}\Pi^0_1$ (which is provable in $\ACA$) at the inductive step together with the fact that $D_{e,i}$ is infinite since $e \in T_i$.  We let $g_i(x,e) = \sigma(x)$ for this $\sigma$.  Having thus defined $g_0$ and $g_1$, let$$Y = \{\langle s, e, i \rangle \in \N^2 \times \{0,1\}: (\exists x)[g_i(x,e) = s]\}.$$Suppose now that $f \leq_T X$ is a computable stable coloring, and fix $e$ such that $f = \Phi^X_e$.  Let $i < 2$ be such that $\lim_s f(x,s) = i$ for infinitely many $x$.  Then $e \in T_i$, so $H = \{s : \langle s, e, i \rangle \in Y \} \leq_T Y$ is infinite and homogeneous for $f$ by construction.
%

We now prove (2). By relativizing Corollary 5.1.7 of Mileti \cite{M}, we get that for any set $X \ngeq_T \emp'$ there is set $Y \geq_T X$ such that $Y \ngeq_T \emp'$ and $Y$ is s-Ramsey relative to $X$ (i.e., computes an infinite homogeneous set for every $X$-computable stable coloring).  Iterating, we thus obtain a sequence $Y_0 \leq_T Y_1 \leq_T \cdots$ such that $Y_e \ngeq_T \emp' $ and $Y_{e+1}$ is \mbox{s-Ramsey} relative to $Y_e$ for every $e$.  Then the ideal $\{ S : (\exists e)[S \leq_T Y_e]\}$ is clearly an $\omega$-model of $\SRAM$ containing no set of degree $\0'$, and hence not a model of $\ACA$.  That $\SRT^2_2$ does not imply $\SRAM$ is because the former has an $\omega$-model consisting entirely of low$_2$ sets by relativizing and iterating Theorem \ref{thm_cjs}, whereas the latter does not by Theorem \ref{thm:joelims}~(2).
\end{proof}

The next result establishes a certain degree of similarity between $\ASRT^2_2$ and $\SRT^2_2$.  In particular, we see that $\ASRT^2_2$ is not overly weak by comparison with at least some of the principles studied in conjunction with $\SRT^2_2$.  The proof resembles that of Theorem 2.4 of \cite{HJKLS} in that it uses the result that every effectively immune set computes a DNR function (see \cite{J2}, p. 199)).  Here we also need the fact, due to Ku\v{c}era, that every 1-random set is effectively bi-immune (\cite{Ku}, Theorem 6).

\begin{prop}\label{prop:wklnoasrt}
Over $\RCA$, $\ASRT^2_2$ implies $\DNR$ but is not implied by $\WKL$.
\end{prop}

\begin{proof}
For the implication, we give only an argument for $\omega$-models, as it, and all the results it employs, admit straightforward formalization in $\RCA$. So let $\M$ be an $\omega$-model of $\ASRT^2_2$ and fix $X \in \M$.  Fix $u$ as in the proof of Proposition \ref{lem:uniform}, let $\widetilde{M} = \Phi^{X'}_u$, and let $\{\widetilde{M}_s\}_{s \in \omega}$ be an $X$-computable approximation of $\widetilde{M}$, sped up to ensure that $2 \widetilde{M}_s(\sigma) = \widetilde{M}_s(\sigma0) + \widetilde{M}_s(\sigma1)$ for all $\sigma$ and $s$.  If we define $M$ by $M(\sigma,s) = \widetilde{M}_s(\sigma)$ for all $\sigma$ and $s$, then $M \in \M$ and is a martingale approximation, so there exists a stable $X$-computable coloring $f \in \M$ and an infinite set $H \in \M$ such that $M$ does not succeed on $f$ and $H$ is homogeneous for $f$.  If we let $A = \{x : \lim_s f(x,s) = 1 \}$ then $\widetilde{M}$ does not succeed on $A$, so $A$ is $X$-random and hence 
%By Theorem 6~(1) of Ku\v{c}era \cite{Ku},
effectively bi-immune relative to $X$.  Then $H$, being an infinite subset or cosubset of $A$, is effectively immune relative to $X$, and so computes a DNR$^X$ function $g \in \M$.

For the nonimplication, recall that for every incomplete $\Delta^0_2$ PA degree $\bf d$ there exists an $\omega$-model of $\WKL$ consisting only of sets of degree below $\bf d$ (this is easily constructed using the fact that the PA degrees are dense; see Simpson \cite{Si2}, Theorem 6.5).  Let $\M$ be any such model.  By Theorem \ref{dam1}, $\bf d$ is not almost s-Ramsey, and so there is a $\Delta^0_2$ martingale $\widetilde{M}$ which succeeds on every $\Delta^0_2$ set containing an infinite subset or cosubset of degree at most $\bf d$.  Let $\{\widetilde{M}_s\}_{s \in \omega}$ be a (suitably sped up) computable approximation to $\widetilde{M}$, and define a martingale approximation  $M \in \M$ from it as above.  Since all stable colorings in $\M$ that have an infinite homogeneous set in $\M$ have one of degree below $\bf d$, it follows that $M$ succeeds on them all.  Thus, $\M$ is not a model of $\ASRT^2_2$.
\end{proof}

It follows that neither $\DNR$ nor $\COH$ imply $\ASRT^2_2$ either, the latter because $\COH$ does not imply $\DNR$ by Theorem 3.7 of \cite{HJKLS}.

In view of the remarks made at the beginning of the section, it is natural to ask whether $\ASRT^2_2$ implies $\WKL$ or  $\COH$ (the preceding proposition makes the first of these at least plausible).  We conclude this section by giving negative answers to both questions.

\begin{prop}
Over $\RCA$, $\ASRT^2_2$ does not imply $\WKL$.
\end{prop}

\begin{proof}
Let $L$ be a given low 1-random set, and let $e \in \omega$ be given.  If $\Phi^{\emp'}_e$ is a total martingale, let $M$, $A$, $B$ and $C$ be as in the proof of Proposition \ref{lem:uniform} with $i$ a lowness index for $L$.  Then $A = B \oplus C$, the set $L \oplus B$ is low, and $B \notin S[\Phi^{\emp'}_e]$.  Furthermore, $A$ is not in $S[M]$ and is therefore $L$-random, so, by van Lambalgen's theorem relative to $L$, $B$ is $L$-random too.  Since $L$ is 1-random, another application of van Lambalgen's theorem yields that $L \oplus B$ is $1$-random.  By iterating, we can thus obtain an increasing sequence of sets $L_0 \leq_T L_1 \leq_T \cdots$ such that each $L_e$ is low, 1-random, and computes a set $B \notin S[\Phi^{\emp'}_e]$ when $\Phi^{\emp'}_e$ is a total martingale.

We let $\M$ be the ideal \mbox{$\{S : (\exists e)[S \leq_T L_e]\}$} and claim first of all that it is a model of $\ASRT^2_2$.  Indeed, suppose that $M \in \M$ is a martingale approximation.  Then $\widetilde{M} : 2^{<\omega} \to \Q^{\geq0}$ defined by $\widetilde{M}(\sigma) = \lim_s M(\sigma,s)$ for all $\sigma$ is a $\Delta^{0,M}_2$ martingale and hence a $\Delta^0_2$ martingale since every element in $\M$ is low.  We can thus fix an $e$ so that $\widetilde{M} = \Phi^{\emp'}_e$.  Then by construction, $L_e$ computes an infinite $\Delta^0_2$ set $B \notin S[\widetilde{M}]$, say with computable approximation $\{B_s\}_{s \in \omega}$.  If we define $f$ by $f(x,s) = B_s(x)$ for all $x < s$, then $f$ is a computable stable coloring, and hence $f \in \M$ and $f \leq_T M$.  Clearly, $M$ does not succeed on $f$ in the sense of Definition \ref{defn:rmmart}, but $B$ computes an infinite homogeneous set $H$ for $f$, which, since $H \leq_T B \leq L_e$, belongs to $\M$.

Now recall that every $\omega$-model of $\WKL$ contains a set of PA degree, and that the class of these degrees is closed upwards (for the former, consider, e.g., the $\Pi^0_1$ class of all $\{0,1\}$-valued DNR functions, and see \cite{DH}, Theorem 1.22.2; for the latter, see \cite{DH}, Theorem 1.21.3).  Also, every 1-random PA degree bounds $\0'$ by the main result of Stephan \cite{St}.  So, as every element of $\M$ is Turing reducible to a low 1-random set, it follows that $\M$ cannot be a model of $\WKL$.
\end{proof}

By Theorems \ref{thm:no_low} and \ref{dam1} respectively, neither $\SRT^2_2$ nor $\ASRAM$ has an $\omega$-model consisting entirely of low sets.  The same is true of $\COH$ because each of its $\omega$-models must contain a p-cohesive set (see \cite{CJS}, p. 27), and each p-cohesive set has jump of degree strictly greater than $\0'$ by Theorem 2.1 of \cite{JSt}.  Hence, we immediately get the following:

\begin{cor}\label{cor:last}
Over $\RCA$, $\ASRT^2_2$ does not imply $\SRT^2_2$, $\ASRAM$, or $\COH$.
\end{cor}

All the relations between the principles studied above are recapitulated in the following diagram (double arrows indicate implications whose reversals are not provable in $\RCA$).
%$$
%{\small \xymatrix{
%& \ACA  \ar@{=>}[d]\\
%& \SRAM \ar@{=>}[dr] \ar@{->}[dl]\\
%\ASRAM \ar@{=>}[dr] & & \SRT^2_2  \ar@{=>}[dl]   \\
%\COH \ar@{->}[r] \ar@/_0pc/[r] |-{\object@{|}} |>{\object@{}} & \ar@{->}[l] \ar@/_0pc/[r] |-{\object@{|}} |>{\object@{}}  \ASRT^2_2  \ar@{=>}[d] \ar@/_0pc/[r] |-{\object@{|}} |>{\object@{}} & \ar@/_0pc/[l] |-{\object@{|}} |>{\object@{}} \WKL\\
%& \DNR \ar@{=>}[d] \\
%& \RCA
%}}
%$$
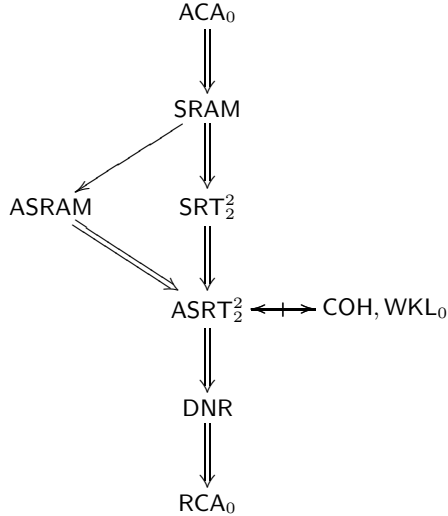
$$
{\small \xymatrix{
& \ACA  \ar@{=>}[d]\\
& \SRAM \ar@{=>}[d] \ar@{->}[dl]\\
\ASRAM \ar@{=>}[dr] & \SRT^2_2  \ar@{=>}[d]   \\
& \ar@{->}[r] \ar@/_0pc/[r] |-{\object@{|}} |>{\object@{}} \ASRT^2_2  \ar@{=>}[d] & \COH, \WKL \ar@{->}[l] \ar@/_0pc/[l] |-{\object@{|}} |>{\object@{}} \\
& \DNR \ar@{=>}[d] \\
& \RCA
}}
$$
%$$
%{\small \xymatrix{
%& &  \ASRAM \ar@{=>}[dr]  &\WKL \ar@/_0pc/[d] |-{\object@{|}} |>{\object@{}} \\
%\ACA \ar@{=>}[r] & \SRAM \ar@{->}[ur] \ar@{=>}[dr]  & & \ASRT^2_2  \ar@/_0pc/[u] |-{\object@{|}} |>{\object@{}}  \ar@/_0pc/[d] |-{\object@{|}} |>{\object@{}} \ar@{=>}[r] & \DNR \ar@{=>}[r] & \RCA\\
%& & \SRT^2_2 \ar@{=>}[ur] & \COH  \ar@/_0pc/[u] |-{\object@{|}} |>{\object@{}}
%}}
%$$
We end by listing a few remaining questions concerning $\ASRAM$ and $\ASRT^2_2$.  Since $\SRT^2_2$ has an $\omega$-model consisting entirely of low$_2$ sets while $\SRAM$ does not, one of the first two would likely be answered by a solution to Question \ref{q:low2}.  The final question concerns the system $\mathsf{WWKL}_0$, introduced in Simpson and Yu \cite{SY}.

\begin{qstn}
Over $\RCA$, does $\ASRAM$ imply $\SRAM$?  Does $\SRT^2_2$ imply $\ASRAM$ or conversely?  Does $\ASRT^2_2$ imply $\mathsf{WWKL}_0$?
\end{qstn}

\noindent $\mathsf{WWKL}_0$ follows from $\WKL$, and so cannot imply $\ASRT^2_2$ by Proposition \ref{prop:wklnoasrt}.  Since the $\omega$-models of $\mathsf{WWKL}_0$ are precisely those that for every set $X$ in them contain also an $X$-random (\cite{AKLS}, Lemma 1.3~(2)), a negative solution to the last question may follow from showing that the collection of $\Delta^0_2$ sets having an infinite subset or cosubset not computing any 1-randoms is not $\Delta^0_2$ null.  It is worth remarking that Kjos-Hanssen \cite{Kj2} (see also \cite{BIRS}, Theorem 7.4) has recently proved the non-effective version of this, showing that almost every infinite subset of $\omega$ has an infinite subset not computing any 1-randoms.

%It would suffice, of course, to show this for the $\Delta^0_2$ sets that are themselves 1-random, and in fact, Kjos-Hanssen (private communication) has shown that every complete $\Delta^0_2$ 1-random has an infinite subset or cosubset not computing any 1-randoms.  However, as mentioned following Proposition \ref{thm:lutzbig}, the complete $\Delta^0_2$ sets are $\Delta^0_2$ null.

%Since the $\omega$-models of $\mathsf{WWKL}_0$ are precisely those which for every set contain also a \mbox{1-random} relative to it (\cite{AKLS}, Lemma 1.3~(2)), a negative solution to the question may follow from showing that the collection of $\Delta^0_2$ sets having a so-called {\em weak} infinite subset or cosubset, i.e. one not computing any 1-randoms, is not $\Delta^0_2$ null.  It would suffices, of course, to show this for the $\Delta^0_2$ 1-random sets, and in fact Kjos-Hanssen (private communication) has shown this for all complete $\Delta^0_2$ 1-randoms.  However, as mentioned following Proposition \ref{thm:lutzbig}, the complete $\Delta^0_2$ sets are $\Delta^0_2$ null.
%However, the complete $\Delta^0_2$ sets are $\Delta^0_2$ null by Theorem 2.3 of \cite{Tq}.
%Kjos-Hanssen \cite[Theorem 3.9]{Kj} has shown that 1-random sets having a weak infinite subset or cosubset exist in general, but for $\Delta^0_2$ sets this is only known (Kjos-Hanssen, private communication) for the complete sets.

\bibliographystyle{plain}
\bibliography{srtmeasbib}

\end{document}